\documentclass[11pt,letter]{amsart}
\usepackage[utf8]{inputenc}

\usepackage{amsmath,amsthm,verbatim,amssymb,amsfonts,amscd, graphicx}
\usepackage{xcolor}
\usepackage{graphics}
\usepackage{mathtools}

\usepackage{euscript, enumerate}
\usepackage{url,hyperref}
\mathtoolsset{showonlyrefs}

\DeclareFontFamily{U}{matha}{\hyphenchar\font45}
\DeclareFontShape{U}{matha}{m}{n}{
      <5> <6> <7> <8> <9> <10> gen * matha
      <10.95> matha10 <12> <14.4> <17.28> <20.74> <24.88> matha12
      }{}
\DeclareSymbolFont{matha}{U}{matha}{m}{n}
\DeclareFontSubstitution{U}{matha}{m}{n}

\DeclareFontFamily{U}{mathx}{\hyphenchar\font45}
\DeclareFontShape{U}{mathx}{m}{n}{
      <5> <6> <7> <8> <9> <10>
      <10.95> <12> <14.4> <17.28> <20.74> <24.88>
      mathx10
      }{}
\DeclareSymbolFont{mathx}{U}{mathx}{m}{n}
\DeclareFontSubstitution{U}{mathx}{m}{n}

\DeclareMathDelimiter{\vvvert}{0}{matha}{"7E}{mathx}{"17}

\newcommand{\mel}{\MoveEqLeft}
\newcommand{\p}{\partial}

\newcommand{\la}{\langle}
\newcommand{\ra}{\rangle}

\newcommand{\e}{\epsilon}
\newcommand{\eps}{\epsilon}

\newcommand{\be}{\begin{equation}}
\newcommand{\ba}{\begin{aligned}}
\newcommand{\bee}{\begin{equation*}}
\newcommand{\ee}{\end{equation}}
\newcommand{\ea}{\end{aligned}}
\newcommand{\eee}{\end{equation*}}
\newcommand{\bea}{\begin{equation} \begin{aligned} }
\newcommand{\eea}{\end{aligned}\end{equation}}
\newcommand{\R}{{\mathbf R}}

\DeclareMathOperator{\dist}{dist}

\newcommand{\grad}{\nabla}
\newcommand{\laplace}{\Delta}
\newcommand{\N}{\mathbf{N}}
\renewcommand{\div}{\grad\cdot}

\theoremstyle{plain}
\newtheorem{theorem}{Theorem}[section]

\newtheorem{lemma}[theorem]{Lemma}

\theoremstyle{remark}

\theoremstyle{definition}
\newtheorem{definition}[theorem]{Definition}

\numberwithin{equation}{section}

\renewcommand{\L}{\mathcal{L}}

\begin{document}
\title[Fast diffusion asymptotics on a bounded domain]{Finite-dimensional leading order dynamics for the fast diffusion equation near extinction}
\date{\today}
\author{Beomjun Choi}
\address{BC: Department of Mathematics, POSTECH, Pohang, Gyeongbuk, South Korea}
\email{bchoi@postech.ac.kr}

\author{Christian Seis}
\address{CS: Institut f\"ur Analysis und Numerik, Universit\"at M\"unster, M\"unster, Germany}
\email{seis@uni-muenster.de}

\begin{abstract}
The fast diffusion equation is analyzed on a bounded domain with Dirichlet boundary conditions, for which solutions are known to extinct in finite time. We construct invariant manifolds that provide a finite-dimensional approximation near the vanishing solution to any prescribed convergence rate.
\end{abstract}
\maketitle  

\section{Introduction}
We study the vanishing behavior of non-negative solutions to the fast diffusion equation on a bounded smooth domain $\Omega$ in $\R^N$ with Dirichlet boundary condition,  that is,  
\begin{equation}\label{20}
\begin{aligned}\partial_\tau w - \laplace w^m&=0 \quad \mbox{in }\Omega,\\
w &=0\quad \mbox{on }\partial \Omega.
\end{aligned}
\end{equation}
For nonlinearity exponents $m\in(0,1)$, the Dirichlet  condition turns the diffusion flux $mw^{m-1}$ singular on the boundary, which has the effect that for bounded solutions the mass escapes entirely in finite time $T>0$ \cite{Sabinina62,Sabinina65,BerrymanHolland80}. 

We will restrict our attention to the Sobolev supercritical range  $m\in (\frac{N-2}{N+2},1)$, and thus, for dimensions $N>2$ we penetrate the so-called very fast diffusion regime $m\in(0,\frac{N-2}N)$. In the latter, certain integrability conditions on the initial data are needed in order to produce bounded solutions. We refer to Chapter 4 of the recent survey paper \cite{BonforteFigalli24} for a discussion. 

  Close to the extinction time $T$,  the dynamics are essentially of separated-variables type, and we consider the accordingly rescaled solutions
\[
w(\tau,x) = \left((1-m)(T-\tau)\right)^{\frac1{1-m}}v^{\frac1m}(t,x),\quad t = \frac{m}{1-m}\log \frac{T}{T-\tau}.
\]
Introducing the reciprocal exponent $p=\frac1m \in (1,\frac{N+2}{N-2})$, the quasi-linearity in the  fast diffusion equation   \eqref{20} is translated  from the spatial derivatives  to  the time derivative, and the resulting equation is reminiscent of the semi-linear heat equation,
\begin{equation}\label{21}
\begin{aligned}
\frac1p \partial_t v^p - \laplace v &= v^p \quad \mbox{in }\Omega,\\
v&=0\quad \mbox{on }\partial \Omega.
\end{aligned}
\end{equation}
It was proved by Berryman and Holland \cite{BerrymanHolland80} that solutions to this rescaled equation approach in infinite time a stationary solution to the semi-linear elliptic equation
\begin{equation}\label{22}
\begin{aligned}
-\laplace V &=V^p  \quad \mbox{in }\Omega,\\
V&=0\quad\mbox{on }\partial \Omega.
\end{aligned}
\end{equation}
In fact, in spite of a possible non-uniqueness of the elliptic problem itself \cite{GidasNiNirenberg79}, Feireisl and Simondon showed that the dynamics \eqref{20} pick a \emph{unique}  solution of \eqref{22}, which is determined by the initial datum $v(t=0)$ (or $w(\tau=0)$, respectively) alone  \cite{FeireislSimondon00}.

Later on, Bonforte, Grillo and V\'azquez \cite{BonforteGrilloVazquez12} proved that the convergence to stationarity proceeds uniformly in the relative error $h = \frac{v-V}V$, that is,
\begin{equation}\label{23}
\lim_{t\to \infty} \|h(t)\|_{L^{\infty}(\Omega)}=0.
\end{equation}
Moreover, they provide sub-optimal estimates on the rate of convergence in the entropy sense for exponents $p$ close to $1$. Notice that the relative error is a natural quantity to consider because it divides out the degeneracy at the boundary, which is enforced by the Dirichlet conditions.  Instead, as we will see in \eqref{eq-relativeerror}, the degeneracy is transferred into the  linear operator.

Establishing the sharp rate of convergence has been a problem of considerable interest for many years, and it was only very recently (partially) successfully  settled. Indeed, Bonforte and Figalli \cite{bonforte2021sharp} compute the sharp rates of convergence ---  at least for $C^{2,\alpha}$ generic domains including the ball, and the authors of the present paper together with McCann \cite{ChoiMcCannSeis22} extend the results (using a completely different approach)  by formulating a dichotomy: Either the relative error decays at least exponentially  in time with rate given by the smallest stable eigenvalue, or the decay is algebraically  $O(1/t)$ or slower. In fact, exponential convergence can be established if the limiting solution $V$ is \emph{ordinary}, that is, if the manifold generated by all (weak) solutions of \eqref{22} has near $V$ the dimension of the kernel of the linear operator $\L $ associated to the dynamics of $h$ (see below), and if the  error relative to $V$ embeds this manifold differentiably into $L^{p+1}(\Omega)$, which is the Lebesgue space into which $W^{1,2}(\Omega)$ embeds compactly precisely if $p < \frac{N+2}{N-2}$. Notice that the critical exponent $\frac{N+2}{N-2}$ distinguishes good from bad behavior in the treatment of the semi-linear elliptic equation \eqref{22}, the rescaled equation \eqref{21} or the semi-linear heat equation. Again, we refer to \cite{BonforteFigalli24} for more information.

To sensitize the reader to  difficulties that had to one has to face when studying the convergence rates for the large-time dynamics \eqref{23}, we consider the evolution equation for the relative error, 
  \bea \label{eq-relativeerror}  \p_t h +  \L h =  \mathcal{M} (h)\quad \mbox{in }\Omega,\eea 
 where $\L $ is the linear operator relative to $V$, that is,
 \begin{equation}\label{eqL}
 \begin{aligned}
 \L  h &= -\frac{1}{ V^{p}} \Delta (hV) - p h\\
& = -V^{1-p}\laplace h - 2V^{-p} \grad V\cdot \grad h - \left({p-1}\right)h\\
& = -V^{-1-p}\div(V^2 \grad h) - \left({p-1}\right)h,
 \end{aligned}
 \end{equation}
{and $\mathcal{M} (h)$ is the nonlinearity, given by}
\begin{equation} \ba 
\label{32a}
 \mathcal{M } (h) &= \frac 1{(1+h)^{p-1}} \left((1+h)^p-1-ph\right) + \left(1- \frac1{(1+h)^{p-1}} \right)L h. \\
&=\frac {\left((1+h)^p-1-ph\right) }{(1+h)^{p-1}} - V^{-1-p} \nabla \cdot (V^2 (1-\frac{1}{(1+h)^{p-1}}) \nabla h )\\
&+ (p-1) \frac{V^{1-p}}{{(1+h)}^p} |\nabla h|^2 -(p-1) (1-\frac{1}{(1+h)^{p-1}} )h   
.\ea \end{equation}
 
Observe that we may write $\mathcal{M} (h)$ in a different form $\mathcal{N}(h)$, which involves $h$ and $\partial _t h$, i.e., for any solution $h$ to \eqref{eq-relativeerror}, it holds that
 \bea \label{eq-problemrelative}
 \mathcal{N} (h) &:= (1+h)^p - 1 - p h  + \big(1-(1+ h)^{p-1} \big) {p} \partial_t h = \mathcal{M} (h) 
 .\eea
Notice that the linear operator $\L $ is self-adjoint in the weighted Lebesgue inner product generating $L^2_{p+1}:=L^2(\Omega,V^{p+1}dx)$, and it is thus invertible \emph{without} imposing boundary conditions on $\partial \Omega$, see also
 Theorem \ref{thm-truncated} below.
 
It is readily checked that  $\L $ admits the  negative (unstable) eigenvalue $1-p$, which corresponds to constant eigenfunctions.  The occurrence of a negative eigenvalue and thus the presence of potentially unstable modes rules out  soft  arguments that are nowadays standard in the derivation of  equilibration rates. Moreover, if $\L $ has a non-trivial kernel, neutral (zero) modes may potentially effect the decay behavior.

The strategy of Bonforte and Figalli to partially overcome these challenges relies on proving that the unstable modes corresponding to negative eigenvalues cannot be active during the evolution in light of Feireisl and Simondon's convergence \cite{FeireislSimondon00} and of Bonforte, Grillo, and V\'azquez' improvement \cite{BonforteGrilloVazquez12}. Focusing on generic domains on which the neutral  modes are absent  \cite{SautTemam79}, Bonforte and Figalli are able to show that the relative error is uniformly decaying with a rate $\lambda$ not smaller than the first positive eigenvalue,
\begin{equation}\label{24}
\|h(t)\|_{L^{\infty}} \lesssim e^{-\lambda t}\quad \mbox{as }t\to \infty.
\end{equation}
Their derivation was subsequently simplified by Agaki \cite{Akagi21+}, though the comparison of his bounds in terms of an energy rather than   the  relative error exploits Jin and Xiong's recent boundary regularity theory \cite{JinXiong22+}.
On arbitrary smooth domains, however, the kernel is not necessarily trivial,  hence the limiting stationary solution is in general not  isolated, and thus, the proposed method is not applicable.

Finally, Jin and Xiong established unconditionally that the convergence is at least algebraic (though with an non-explicit power) \cite{JinXiong20+}. 

In \cite{ChoiMcCannSeis22}, in which the authors with McCann bridge the wide gap between upper exponential  bounds and algebraic lower bounds on the rate of convergence, a different approach is developed. It reduces the derivation of decay estimates to the analysis of a dynamical system that fits into the setting of Merle and Zaag's ODE lemma \cite{MerleZaag98}. Exploiting the uniform convergence in \eqref{23}, the lemma implies that the dynamics are eventually dominated either by the neutral or by the stable modes. The main (technical) work is spent in \cite{ChoiMcCannSeis22} to obtain new smoothing estimates, that allow to uniformly control  temporal and (tangential) spatial derivatives of the solution at positive times by a weighted $L^2$ norm of the initial datum.

If a solution decays exponentially, apart from providing the leading order estimates on the decay of the relative error \eqref{24}, it is showed in \cite{ChoiMcCannSeis22} that the large-time behavior is described by the linear dynamics up to an error of the order $e^{-2\lambda t}$ which is generated by the quadratic behavior of $N(h)$ in \eqref{eq-problemrelative} near the limit $h=0$. More precisely, if the solution is decaying at rate $\lambda$ \eqref{24}, it can be approximated by an eigenmode expansion corresponding to the eigenvalues in the interval $(\lambda, 2\lambda)$. If a solution decays non-exponentially, it is eventually dominated by the neutral mode and the dynamics of dominating modes becomes non-linear in an essential way. No precise asymptotics, such as the rate of convergence and the dominating profile, is known except of the lower and upper algebraic bounds  provided by \cite{ChoiHaslhoferHershkovits18+} and \cite{JinXiong22+}, respectively. The problem of determining the exact algebraic or near-algebraic decay behavior in general setting is still unsettled for equations with gradient flow structures. In  recent work \cite{choi2023asymptotics}, the first named author and Hung address a higher order asymptotics for a class of equations algebraically converging to smooth compact limits without boundary.

The goal of the present paper is to further refine the information on the large time dynamics. Our main result provides an unconditional convergence of  the decaying solution \eqref{23} towards an approximate solution that lives on a  finite-dimensional  manifold.  More precisely,  denoting the (discrete \cite[Lemma 2.1]{bonforte2021sharp}) spectrum of $\L $ by $\sigma(\L )= \{ \lambda_k \,:\, k \in \mathbb{N}\}$ and assuming that the eigenvalues $\lambda_k$ are strictly increasing, we introduce  for a given integer $K\in \mathbb{N}$ the eigenspace
\be\ba  E_c &= \mathrm{span} \{ \text{eigenfunctions with }\lambda_k \le \lambda_K \} \ea \ee 
 and consider the orthogonal decomposition 
\be L^2_{p+1} = E_c \oplus E_s ,\ee
so that  $E_s$ is spanned by the eigenfunctions with $\lambda_k >\lambda_K$. Let $\lambda_-\in (\lambda_K,\lambda_{K+1})$ be fixed.

\begin{theorem} [Finite dimensional approximation] \label{thm-finiteapprox} Let $h(t)$ be a solution to \eqref{eq-relativeerror}  that converges to $0$ as $t\to \infty$ in $L^\infty (\Omega)$. For every $\eps \in(0,\eps_0)$, there exist a Lipschitz continuous function $\theta_{\eps}: E_c\to E_s$ with $\theta_{\eps}(0)=0$,  a time $t_0>0$ and an evolution $\tilde h_{\eps}(t)\in \mbox{graph}_{E_c}\theta_{\eps} $   for $t\ge t_0$ such that
\[\Vert h(t)-\tilde h_\eps  (t)\Vert_{L^2_{p+1}} \lesssim   e^{-\lambda_- (t-t_0)} \text{ for all } t\ge t_0 .\]
Moreover, $\tilde h_{\eps}(t)$ is a solution to \eqref{eq-relativeerror} provided that
\begin{equation}\label{100}
\|\tilde h(t)\|_{L^{\infty}} \le \eps,\quad \|V\grad\tilde h(t)\|_{L^{\infty}}\le \eps.
\end{equation}
\end{theorem}
In order to understand the meaning of the second estimate in \eqref{100}, we recall that the  decay  of the stationary solution towards the domain boundary is essentially linear, that is, there is a constant $C$ such that
\begin{equation}
\label{110}
\frac1C\dist(x,\partial \Omega) \le V(x) \le C\dist(x,\partial \Omega),
\end{equation}
for any $x\in \Omega$, see, for instance, Theorem 1.1 in \cite{DiBenedettoKwongVespri91} or Theorem 5.9 in \cite{BonforteGrilloVazquez13}. Hence, $\tilde h$ is a solution to the original equation \eqref{eq-relativeerror} if its amplitude is small and if it is not growing too much towards the boundary. The smallness in amplitude applies (at least, after a waiting time) thanks to the convergence in \eqref{23}. In Lemma \ref{L11} below, we will recall that also the second estimate is eventually satisfied.

In the construction of $\theta_{\eps}$ below we will see that the \emph{center manifold}
\[
W^c_{\eps} = \mbox{graph}_{E_c}\theta_{\eps}
\]
is invariant under the evolution of a suitably truncated version of the dynamics \eqref{eq-relativeerror}. The truncation is of the kind that it is inactive for solutions to the original equation after a certain waiting time, provided that the solution is uniformly small in the sense of \eqref{100}, cf.~Lemma \ref{L11} below. The need of the truncation in the formulation of the theorem is a consequence of   a lack of suitable $L^q$ regularity estimates for the linear version of \eqref{eq-relativeerror} that we believe to hold true but whose proof is beyond the scope of the present paper. Nonetheless, the statement contributes to our understanding of the large-time asymptotics of the relative error. Choosing $\lambda_K=0$, we have the following dichotomy according to our earlier findings \cite{ChoiMcCannSeis22}: If the solution vanishes at an  exponential rate, the approximate solution is trivial, $\tilde h=0$. Otherwise, if the solution is at most algebraically decaying,  there is  {an approximate} solution in the finite-dimensional space that describes the given solution up to an exponentially decaying error.  It is the restriction to the finite-dimensional manifold that we consider as the major progress provided in this contribution.

In principle, that choice of $K$ in the statement of the theorem is arbitrary, and we could also choose $\lambda_K>0$. In this case, we could gain some information on the higher-order asymptotics similar to, e.g., \cite{Seis15+,SeisWinkler22+}. Because in the general setting under consideration eigenfunctions are not explicitly known, pursuing this direction  is a rather abstract maneuver that we prefer to avoid. 

\emph{The paper is organized as follows:} In Section \ref{S2}, we introduce the truncated model and discuss its well-posedness. Section \ref{S3} then provides the construction of the invariant manifolds and the proof of Theorem \ref{thm-finiteapprox}.

\section{Truncated problem}\label{S2}
The construction of invariant manifolds necessitates the availability of a semi-flow that is defined \emph{globally} on the underlying Banach space. It is thus customary to suitably truncate the nonlinearity in \eqref{eq-relativeerror}. We will do it in such a way that the divergence structure of the leading order term in partially retained. 

For a standard cut-off function $\eta: \mathbb{R} \to \mathbb{R}_{\ge 0}$ that has a support on the interval $[-2,2]$ and satisfies $\eta=1$ on the smaller interval $[-1,1]$ and a positive real $\eps$,  let us introduce the two truncations  
 \be \label{eq-trunccutoff0}\eta^\e_0=\eta^\e _0[h]:= \eta(\frac{h}{\e})\ee  and \be \label{eq-trunccutoff1}\eta^\e_1=\eta^\e _1 [h]:= \eta (\frac{h}{\e}) \eta (\frac{V \p_1 h}{\e}) \cdots  \eta (\frac{V \p_N h}{\e}).\ee 
With that we consider our evolutionary problem with  a truncated nonlinearity,
\be\label{eq-trunc} \partial_t h + \L  h= \mathcal{M}^\e  [h],\ee
where we have set
\be\ba  \label{eq-Mtrunc} \mathcal{M}^\e   [h]& =  \eta^\e_0\left[ \frac {\left((1+h)^p-1-ph\right) }{(1+h)^{p-1}}-(p-1) (1-\frac{1}{(1+h)^{p-1}} )h\right] \\
& - V^{-1-p} \nabla \cdot (V^2  (1-\frac{1}{(1+h)^{p-1}})\eta^\e _1 \nabla h ) + (p-1) \eta^\e _1 \frac{V^{1-p}}{{(1+h)}^p} |\nabla h|^2 .\ea \ee  
The choice of our truncations guarantees that for functions that are small in amplitude, $|h|\le \eps$ and which are at most mildly increasing towards the domain boundary in the sense $V|\grad h|\le \eps $, it holds that $\mathcal{M}^{\e} [h] = \mathcal{M} [h]$, and thus, the original equation \eqref{eq-relativeerror} and its truncated variants \eqref{eq-trunc} are both equivalent. This observation is the origin of \eqref{100}.

Before turning to the analysis of the truncated equation, we remark that our decaying solution \eqref{24} to the original equation indeed  solves the truncated equation after a possible waiting time.

\begin{lemma}\label{L11}
For any $\eps>0$, there exists a constant $\eps_*\le \eps$ with the following property: If $h$ is a solution to the original equation \eqref{eq-relativeerror} with $\|h(t)\|_{L^{\infty}}\le \eps_*$ for all $t\ge 0$, then
\[
\|V\grad h \|_{L^{\infty}}\le \eps,
\]
for all $t\ge 1$.
\end{lemma}

In view of \eqref{24}, the smallness condition on $h$ in the hypothesis of the lemma is always true for sufficiently large times.

\begin{proof}
The estimate in the statement of the lemma was essentially already established in our previous work \cite{ChoiMcCannSeis22}. Indeed, \cite[Corollary 5.12]{ChoiMcCannSeis22} provides the estimates near the boundary (after straightening the boundary with the help of a diffeomorphism), while estimates in the interior result from standard parabolic estimates. 
\end{proof}

We start now by discussing the well-posedness of the truncated problem.

\begin{theorem}\label{thm-truncated}There exists $\e_0 >0$ such that for all $0<\e \le \e_0$ and $h_0 \in L^2_{p+1}$, there exists a unique weak solution $h_\e \in C^0_{loc} (L^2_{p+1}) \cap L_{loc}^2 (H^1_2)$ to the truncated problem \be \label{eq-trunc-eq} \begin{cases} \ba  \p_t h+ \L h & = \mathcal{M}^\e  [h],  \\ h(0)&=h_0 .
 \ea \end{cases} \ee  
Moreover, there holds a stability estimate
\bea \label{eq-stability}
\sup_{[0,T]} \Vert h-\tilde h\Vert_{L^2_{p+1}}^2 + \int_0^T \Vert \nabla h& -\nabla \tilde h \Vert_{L_2^2}^2\, dt  \le e^{CT} \Vert h_0 -\tilde h_0\Vert^2_{L^2_{p+1}} , \eea
for $h(t)$ and $\tilde h(t)$ solutions to \eqref{eq-trunc-eq} with initial datum $h_0$ and $\tilde h_0$, respectively.
\end{theorem}

Here, we have used $L^2_2$ to denote  the weighted space $L^2(\Omega, V^2dx)$, and the Sobolev space $H^1_2$ is the one for which both the function $h$ and its gradient $\grad h$ belong to $L^2_2$. However, the precise weight in the lower order term is arbitrary as can be seen in the following Hardy-type inequality, that will be used during the proof,
\be \label{eq-hardy}  \Vert h \Vert _{L^2} \lesssim \Vert h \Vert _{L^2_{p+1}} + \Vert \nabla h \Vert _{L^2_2}, \ee 
and which was established in Lemma 5.2 of \cite{ChoiMcCannSeis22}. Actually, via an application of Young's inequality, the weight on the $L^2$ part can indeed be chosen arbitrary large. 

From now on, we use $\Vert \cdot \Vert$ and $\la \cdot, \cdot \ra$ to denote the norm and the inner product in $L^2_{p+1}$.

\begin{proof}[Proof of Theorem \ref{thm-truncated}] 

The proof will be based on a contraction mapping principle argument. It is enough to prove the statement on the finite time interval $[0,1]$, because then a global solution can be constructed iteratively. Notice that the stability estimate \eqref{eq-stability} yields the a priori bound
\[
\sup_{[0,T]} \Vert h \Vert^2 + \int_0^T \Vert \nabla h\Vert_{L_2^2}^2\, dt  \le e^{CT} \Vert h_0  \Vert^2 , 
\]
because the constant function $\tilde h\equiv 0$ is a trivial solution of \eqref{eq-trunc}.

\noindent \emph{Step 1. The linear problem.} Recall that the linear operator $\L $ is self-adjoint in $L^2_{p+1}$ as discussed in Section 2 of \cite{ChoiMcCannSeis22}. By the standard Galerkin approximation scheme and energy estimate ({\it a.k.a.}~the integration by parts), for given $g \in L^\infty L^2_{p+1} \cap L^2H^1_2$ and $h_0 \in L^2_{p+1}$, we obtain unique $h\in L^\infty L^2_{p+1} \cap L^2H^1_2$ solving \be \label{111}\begin{cases} \ba  \p_t h+ \L h & = \mathcal{M}^\e  [g],  \\ h(0)&=h_0 ,
 \ea \end{cases} \ee in the weak sense, and there holds an energy estimate 
\[\Vert h \Vert _{L^\infty L^2_{p+1}} + \Vert \nabla h \Vert _{L^2L_2^2} \lesssim \e ( \Vert g \Vert_{L^\infty L^2_{p+1}}+ \Vert \nabla g \Vert _{L^2L_2^2})+ \Vert h_0 \Vert .\]
Here, the prefactor $\eps$ on the right-hand side comes from the (at least) quadratic nonlinear behavior of the nonlinearity  $\mathcal{M}^\eps [g]$ and of our choice of truncation, that is, we use the bounds $|g|\le \eps$ and $V|\grad g|\le \eps$ in the support of $\mathcal{M}^{\eps}[g]$.

\medskip

\noindent \emph{Step 2. Stability/contraction of linear problem.} For given source functions $g$, $\tilde g\in L^{\infty}L^2_{p+1}\cap L^2 H^2_2$, and initial data $h_0$, $\tilde h_0 \in L^2_{p+1}$, let $h$, $\tilde h$ be the   unique  solutions $h$ and $\tilde h$ to the corresponding linear problems \eqref{eq-trunc-eq} as constructed in the previous step. Here we show that there holds
\begin{equation} \label{eq-step2}
\begin{aligned}
\mel \Vert h-\tilde h\Vert_{C^0 L^2_{p+1}} + \Vert \nabla h -\nabla \tilde h \Vert_{L^2L_2^2}\\
& \lesssim \e (\Vert g-\tilde g \Vert_{L^\infty L^2_{p+1}} + \Vert \nabla g -\nabla \tilde g \Vert _{L^2L_2^2}) + \Vert h_0 -\tilde h_0\Vert. 
\end{aligned}
\end{equation} 

By testing the equation \eqref{111} for the difference $h-\tilde h$ with $(h-\tilde h)V^{1+p}$, we obtain 
\be \label{eq-difference-energy}\frac12\frac{d}{d t } \Vert h- \tilde h \Vert ^2 + \Vert \nabla h -\nabla \tilde h \Vert ^2_{L^2_2} = \la h-\tilde h, \mathcal{M}^\eps [g]-\mathcal{M}^\eps [\tilde g]\ra +(p-1)\Vert h-\tilde h\Vert ^2 . \ee  
We have to estimate the right-hand side in the sequel, which is, in fact, quite elementary. For instance, to treat   the zeroth-order derivative terms, we notice that for any nonlinear smooth function $f(z)$ such that $f(z) = O(z^2) $ for $|z| \to 0$, it holds that
\[
|\eta_0^{\eps}(z)f(z) - \eta_0(\tilde z)f(\tilde z)| \lesssim \eps |z-\tilde z|.
\]
Indeed, if, for instance $|z|\le 2\eps<|\tilde z|$, we have that
\[
|\eta_0^{\eps}(z)f(z) - \eta_0(\tilde z)f(\tilde z)| = \eta_0^{\eps}(z)f(z)  \lesssim \eps^2 \eta_0^{\eps}(z) = \eps^2 |\eta_0^{\eps}(z) - \eta_0^{\eps}(\tilde z)| \lesssim \eps |z-\tilde z|,
\]
where in the last inequality we have used that the truncation function is Lipschitz with a constant of order $\eps^{-1}$. The other cases where both variables are either smaller or larger than $2\eps$ are even simpler. The estimate of the terms in \eqref{eq-difference-energy} that involve the gradients proceeds very similar in the sense that there are no qualitatively new arguments involved.  For the second-order-derivative term, an additional integration by parts is necessary. We skip the details. Eventually, we arrive at
\[
\left|\la h-\tilde h, \mathcal{M}^\eps [g]-\mathcal{M}^\eps [\tilde g]\ra \right| \lesssim \eps \left(\|g-\tilde g\| +\|g-\tilde g\|_{L^2}+ \|\grad g-\grad\tilde g\|_{L^2_2}\right),
\]
and the Hardy-type estimate in \eqref{eq-hardy} tells us that the $L^2$ term on the right-hand side can be dropped. We insert the resulting estimate into \eqref{eq-difference-energy} to obtain
\begin{align*}
\mel
\frac{d}{d t } \Vert h- \tilde h \Vert ^2 + \Vert \nabla h -\nabla \tilde h \Vert ^2_{L^2_2} \\
&\lesssim \eps \left(\|h-\tilde h\| + \|\grad h-\grad\tilde h\|_{L^2_2}\right)\left(\|g-\tilde g\| + \|\grad g-\grad\tilde g\|_{L^2_2}\right)  + \Vert h-\tilde h\Vert ^2 .
\end{align*}
An integration yields \eqref{eq-step2}. Via Young's inequality the latter becomes
\begin{align*}
\mel
\frac{d}{d t } \Vert h- \tilde h \Vert ^2 + \Vert \nabla h -\nabla \tilde h \Vert ^2_{L^2_2} \\
&\lesssim \eps^2  \left(\|g-\tilde g\|^2 + \|\grad g-\grad\tilde g\|_{L^2_2}^2\right)  + \Vert h-\tilde h\Vert ^2 ,
\end{align*}
and we obtain \eqref{eq-step2} upon integration. Continuity in time follows by a standard argument.

\medskip
\noindent \emph{Step 3. Fixed point argument.} From two steps above, for $h_0 \in L^2_{p+1}$ given, the mapping of solution in \emph{Step 1} 
\be  \label{eq-hcontraction}g \mapsto h(h_0,g)\ee 
is a contraction on $L^\infty L^2_{p+1} \cap L^2 H^1_2$ (here we uses \eqref{eq-hardy} to conclude $h\in L^2H^1_2$) for $\e\le \e_0(T)$. Therefore, there exists unique $h$ solving the nonlinear truncated problem \eqref{eq-trunc-eq}. Via standard observations it finally follows that $t\mapsto h(t)\in L^2_{p+1}$ is continuous.
\end{proof}

%

\section{Invariant manifolds}\label{S3}

In this section, we construct invariant manifold for the truncated equation. Our argument follows a method proposed by Koch \cite{Koch97}, which relies on a time discretization. It is thus enough to study the evolution of the time-one maps. Koch's  method was recently used in \cite{Seis15+,SeisWinkler22+} to construct invariant manifolds in the context of the porous medium equation and the thin film equation both on $\R^N$.

Let us denote the  orthogonal projections onto the eigenspaces $E_c$ and $E_s$ by $P_c$ and $P_s$, respectively. Note that $P_c+P_s=I$.  Let us introduce the restricted linear operators
\be \mathcal{L}_c = \mathcal{L} P_c, \quad \mathcal{L}_s = \mathcal{L} P_s = \mathcal{L} -\mathcal{L}_c, \ee
and the associated semi-group operators,
\be L = e^{-\mathcal{L}}, \quad  L_c=e^{-\mathcal{L}_c}, \quad L_s=e^{-\mathcal{L}_s} . \ee  
If $\Pi_k$ denotes the projection onto $\lambda_k$-eigenspace, $h_k(t):= \Pi_k L^t h_0$ solves the linear equation 
$\partial_t h_k+\lambda_k h_k =0$
and hence
\[h_k(t) = e^{-\lambda_k t} \Pi_k h_0. \]
Note that we have bounds on the operator norms
\begin{equation}\label{500}
\Vert L_c^{-1}\Vert \le e^{\lambda_K} ,\quad \Vert L_s \Vert \le e ^{-\lambda_{K+1}} ,\quad \|L\|\le e^{-\lambda_1}.
\end{equation}
Indeed, this follows by
\bea   \Vert L^{-1}_cL_c h_0 \Vert ^2 &=\Vert P_c h_0\Vert^2  =\sum _{i\le K } \Vert \Pi_k h_0 \Vert^2 =  \sum _{i\le K}e^{2\lambda_k} \Vert h_k(1)\Vert^2 \\ & \le \sum e^{2\lambda_K} \Vert h_k(1)\Vert^2   \le  e^{2\lambda_K}\Vert L_c h_0\Vert ^2  ,\eea and a similar computation for $L_s$. 

Choose two numbers $\lambda_+$ and $\lambda_-$ such that 
\[\lambda_+ <\lambda_1 <\cdots <\lambda_K<\lambda_- < \lambda_{K+1} \]
 and also denote 
\[\Lambda_+=e^{-\lambda_+}, \Lambda_{\max} = e^{-\lambda_1}, \Lambda_c=e^{-\lambda_K} ,  \Lambda _- =e^{-\lambda_-}, \Lambda_s = e^{-\lambda_{K+1}}. \]
Observe we have an order in $\Lambda$
\[ \Lambda_s < \Lambda_-<\Lambda_c<\Lambda_{\max }<\Lambda_+.\]

Let us define the time-$t$-map of truncated problem $S^\e : L^2_{p+1} \to L^2_{p+1}$ by 
$$S^\e _t(h_0) := h(t)$$ where $h(t)$ is the unique weak solution of truncated problem  starting from $h_0$ 
given by Theorem \ref{thm-truncated}. In particular, we denote the time-$1$-map by $S^\eps := S^\eps_1$. 

We may decompose the operator $S^\e _1$ as $$S^\e = L+R^\e.$$  If $\eps$ is small enough, the nonlinear part $R^{\eps}$ is a contraction.
We will omit the super script truncation parameter $\e$ for $S$, $R$ and others whenever there is no confusion. 

\begin{lemma} \label{lemma-R}
Let $\eps$ be the truncation parameter. Then  $R(0)=0$ and there holds
the estimate 
\[ \Vert  R(\hat h_0)-R(\tilde h_0)\Vert \le C \eps \Vert \hat  h_0-\tilde h_0 \Vert .\]
\end{lemma}

\begin{proof}
Let $\hat h(t)$ and $\tilde h(t)$ be solutions to \eqref{eq-trunc-eq} with the initial data $\hat h_0$ and $\tilde h_0$, respectively. Note that $R(\hat h_0)=\hat g(1)$ and $R(\tilde h_0)= \tilde g(1)$ where $\hat g(t)$ is the weak solution to  
\be \label{eq-g} \begin{cases} \ba  \p_t \hat g+ \L  \hat g & = \mathcal{M}^\e  [\hat h]  \\ \hat g(0)&=0 
 \ea \end{cases} \ee  
and $\tilde g(t)$ solve  the corresponding equation with the right-hand side given by $\mathcal{M}^\eps  [\tilde h]$. In view of  the energy estimate in \eqref{eq-step2} it holds that
\[\Vert \hat g-\tilde g\Vert _{C^0  L^2_{p+1}} + \Vert \nabla \hat g -\nabla \tilde g \Vert _{L^2L^2_2} \le C \eps (\Vert \hat  h- \tilde h \Vert _{L^2L^2_{p+1}}+ \Vert \nabla  \hat h -\nabla \tilde h \Vert _{L^2 L^2_2}),\]where the $C^0$ and $L^2$ norms on time variable are computed on $t\in [0,1]$. {Now the statement follows from  estimate \eqref{eq-step2}.} The statement that $R(0)=0$ is a simple consequence of the definition and the fact that $\mathcal{M} ^{\eps}[0]=0$.

\end{proof}

Let us define the convenient norm \[\vvvert h\vvvert := \max ( \Vert P_c h \Vert_{L^2_{p+1}} , \Vert P_s h\Vert_{L^2_{p+1}} ),\]which is equivalent to $\Vert h\Vert_{L^2_{p+1}}$. For a bi-directed sequence of functions in $L^2_{p+1}$, say $\{ h_{k} \} _{k\in \mathbb{Z}}$,  let us define the norm $\vvvert\cdot \vvvert _{\Lambda_-,\Lambda_+}$ and the space
\be  \ell _{\Lambda_-,\Lambda_+} := \{\{h_k \}_{k\in \mathbb{Z}}\,:\,   \vvvert \{h_k \}_{k\in \mathbb{Z}} \vvvert _{\Lambda_-,\Lambda_+}:=\sup _{k\in \mathbb{N}_0} \{ \Lambda _+^{-k} \vvvert h_k\vvvert , \Lambda_-^{k} \vvvert h_{-k} \vvvert  \} <\infty  \}.\ee
Similarly, we define the norm $\vvvert \cdot \vvvert_{\Lambda_-} $ and the space 
\be \ell_{\Lambda_-} :=\{ \{h_k \}_{k\in \mathbb{N}_0} \,:\, \vvvert \{h_k\}\vvvert_{\Lambda_-}:= \sup_{k\in \mathbb{N}_0} \Lambda_-^{-k} \vvvert  h_k \vvvert  <\infty \}  .\ee

\subsection{Center manifolds} 
Define an operator $J: E_c \times \ell _{\Lambda_-,\Lambda_+}\to \ell _{\Lambda_-,\Lambda_+}$ with $J(\cdot,\cdot)= \{ J_k(\cdot,\cdot) \} _{k \in \mathbb{Z}}$ by
  \be \label{eq-defJ} J_k(h_c, \{h_i\}):= \begin{cases} \ba  S( h_{k-1}) &&\text{ if } k\ge 1 \\
 P_s S(h_{-1})+h_c &&\text{ if } k =0 \\
 P_s S(h_{k-1})+L_c^{-1} P_c (h_{k+1} -R (h_k))&&\text{ if }k\le -1 
   \ea \end{cases}\ee  

The evolution on the center manifold is obtained via a fixed point for $J$. Its existence is our first goal.

\begin{lemma} \label{thm-centermfd} For sufficiently small truncation parameter $\eps \le \eps_0$, the mapping $J(h_c,\cdot)$ admits a unique fixed point $J(h_c, \{h_k\})= \{h_{k}\}$ and there holds 
\be \label{eq-centermfdfixedpoint}\vvvert \{h_k\}\vvvert _{\Lambda_-,\Lambda_+} \le  \vvvert h _c \vvvert .\ee 
\end{lemma}

\begin{proof}
Choose $\eps_{gap}$ so that 
\be \label{eq-egap} \min (\frac{\Lambda_{\max} + \e_{gap}}{\Lambda_+}, \frac{\Lambda_s +\e _{gap}}{\Lambda_-}, \frac{\Lambda_-+ \e _{gap}}{\Lambda _c} )=K<1 .\ee  By Lemma  \ref{lemma-R}, we can find a small constant $c$ such that for $\eps <c\eps_{gap}$, the solutions to the truncated problem satisfy 
\be \label{eq-R-eps-gap}\vvvert R^\eps (h)-R^\eps (\tilde h) \vvvert \le \e _{gap} \vvvert h- \tilde h \vvvert. \ee 

First we show a bound  
\bea \label{112}
\vvvert J(h_c, \{ h_k \} ) \vvvert \le \max (\vvvert h_c\vvvert, \vvvert \{h_k\}\vvvert_{\Lambda_-,\Lambda_+}),\eea which in particular implies $J(h_c,\cdot)$ is a self-mapping \[J(h_c,\cdot): B_{\vvvert h_c \vvvert }\to B_{\vvvert h_c \vvvert } .\]
For the sake of simplicity, let us denote $J(h_c, \{ h_k\})= \{J_k\}$.
If $k\ge 1$, using the linear maximal estimates \eqref{500} on $L$ and the bound in \eqref{eq-R-eps-gap} gives \bea \Lambda_+^{-k} \vvvert J_k\vvvert = \Lambda_+^{-k}\vvvert S(h_{k-1})\vvvert  & \le \Lambda_+^{-k} (\vvvert Lh_{k-1}\vvvert + \vvvert R(h_{k-1})\vvvert )\\
& \le {\Lambda_+^{-(k-1)}}\frac{\Lambda_{\max}+\eps_{gap}}{\Lambda_+}  \vvvert h_{k-1}\vvvert \le \vvvert \{ h_l \}\vvvert_{\Lambda_-,\Lambda_+}.\eea  
If $k= 0$, by a similar argument, we find \bea \vvvert J_0\vvvert &\le  \max (\vvvert h_c \vvvert, \vvvert P_s S(h_{-1})\vvvert)\\
 & \le \max   (\vvvert h_c \vvvert,  \frac{\Lambda_s +\eps_{gap}}{\Lambda_-} \vvvert \{ h_l \}\vvvert_{\Lambda_-, \Lambda_+})\\
&\le  \max   (\vvvert h_c \vvvert,  \vvvert \{ h_l \}\vvvert_{\Lambda_-, \Lambda_+}).\eea
Finally, if $k\le -1$, it holds that
\bea \Lambda_-^{-k} \vvvert J_k\vvvert & =\Lambda_-^{-k}\max (\vvvert P_sS(h_{k-1})\vvvert, \vvvert L_c^{-1} P_c (h_{k+1} -R(h_k) ) \vvvert )   \\
& \le \Lambda_-^{-k}\max ( { (\Lambda_s +\e _{gap})} \vvvert h_{k-1}\vvvert   , \frac{\Lambda_- + \e _{gap}}{\Lambda _c}  \vvvert h_k \vvvert )   \\ 
&\le \vvvert \{h_l \}\vvvert _{\Lambda^-,\Lambda_+}.\eea  
By the condition on $\e_{gap}$ in \eqref{eq-egap} this shows the bound.

Next, by the same argument as above, we could have shown that 
\[\vvvert J(h_c,\{h_k\})-J(h_c,\{\tilde h_k \}) \vvvert_{\Lambda_-,\Lambda_+} \le K \vvvert \{h_k\} - \{\tilde h_k \} \vvvert _{\Lambda_-,\Lambda_+} , \] 
and this shows $J(h_c,\cdot)$ is a contraction mapping on $\ell _{\Lambda_-,\Lambda_+}$ and thus there exists a unique fixed point in $\vvvert h_c\vvvert$-ball and it satisfies the estimate \eqref{eq-centermfdfixedpoint} by the virtue of \eqref{112}.
\end{proof}

The fixed point $\{h_k\}$ obtained from Theorem \ref{thm-centermfd} is the unique eternal solution to $S(h_{k-1})=h_k$ on $\ell_{\Lambda_-,\Lambda_+}$ that satisfies $P_c h_0 = h_c$. Indeed, for $k\ge 1$, it holds that $h_{k}= S(h_{k-1})$ by definition. For $k=0$, we observe that $h_0= P_s S(h_{-1})+h_c$, which implies $P_c h_0 = h_c$. For $k\le -1$ it holds $P_s h_{k}=P_s S(h_{k-1}) $ by definition and $P_c h_{k} = L_c^{-1}(P_c(h_{k+1} -R(h_k))$ implies that 
\[L_c h_k +P_cR(h_k) = P_c S(h_k)= P_c(h_{k+1}).\]
Combining all above, $S(h_{k-1})= h_k$ for all $k\in \mathbb{Z}$ and $P_c(h_0)=h_c$. The uniqueness follows by the uniqueness of the fixed point on $\ell _{\Lambda_-,\Lambda_+}$.

From now on, we denote the fixed point map $h_c \mapsto \{h_k\}$ by $ \Theta=\Theta_{\Lambda_-,\Lambda_+}: L^2_{p+1} \to \ell _{\Lambda_-,\Lambda_+} $. Theorem \ref{thm-centermfd} directly says $\vvvert \Theta (h_c )\vvvert_{\Lambda_-,\Lambda_+} \le \vvvert h_c \vvvert $.

\begin{lemma}\label{lem-lipTheta}$ \Theta$ is a Lipschitz map from $(L^2_{p+1}\cap E_c, \vvvert \cdot \vvvert)$ to $(\ell _{\Lambda_-,\Lambda_+}, \vvvert \cdot \vvvert_{\Lambda_-,\Lambda_+})$ with $\mathrm{Lip}(\Theta)\le \frac1{1-K}$. 
\begin{proof}
We have shown that $J(h_c,\cdot)$ is a contraction mapping for small $\eps<\eps_0$ with Lipschitz constant $K<1$. For $h_c$ and $\tilde  h_c$ in $E_c$,  
\bea &\vvvert \Theta (h_c ) - \Theta(\tilde h_c ) \vvvert_{\Lambda_-,\Lambda_+} \le \vvvert J(h_c, \Theta (h_c))- J(\tilde h_c, \Theta (\tilde h_c)) \vvvert \\
& \le \vvvert J(h_c, \Theta(h_c))-J(h_c,\Theta (\tilde h_c))\vvvert + \vvvert   J(h_c,\Theta (\tilde h_c))-J(\tilde h_c,\Theta (\tilde h_c))\vvvert  \\
&\le K \vvvert \Theta(h_c)-\Theta(\tilde h_c )\vvvert + \vvvert h_c-\tilde h_c \vvvert ,\eea 
where the last inequality follows by the definition of $J$ in  \eqref{eq-defJ}. Now this shows $\Theta$ is $\frac{1}{1-K}$-Lipschitz map. 
\end{proof}
\end{lemma}

Suppose $\Theta (f)= \{\Theta_k(f) \}_{k\in \mathbb{Z}}$. We define $\theta : E_c\to E_s$ by 
\[\theta (f) := P_s \Theta_0 (f) = \Theta_0 (f) -f .\]

By the virtue of Lemma \ref{lem-lipTheta}, we have $\mathrm{Lip}(\theta)\le \frac{1}{1-K}$. Here, we want to improve on the Lipschitz constant.

\begin{lemma}\label{lem-liptheta}The map $\theta: (E_c\cap L^2_{p+1},\vvvert \cdot \vvvert )\to (E_s\cap L^2_{p+1},\vvvert \cdot \vvvert) $
is a Lipschitz with $\mathrm{Lip}(\theta) \lesssim \e _{gap}$.
\end{lemma}

\begin{proof} 
 For given $g$ and $\tilde g$ in $E_c$, let $\{h_k\}=\Theta(g)$ and $\{ \tilde h_k\}= \Theta(\tilde g )$. Since \bea \theta(g)-\theta(\tilde g) &=  P_s (S(h_{-1})-S(\tilde h_{-1})) \\
&= P_s (L(h_{-1}-\tilde h _{-1}))+P_s(R(h_{-1})-R(\tilde h_{-1})) ,\eea
we have that
\bea \vvvert \theta(g)-\theta(\tilde g ) \vvvert \le \Lambda_s \vvvert P_s  h_{-1} - P_s \tilde h_{-1} \vvvert +C\eps _{gap}\vvvert h_{-1}-\tilde h_{-1} \vvvert  .  \eea
By iterating this estimate $k$ times, we obtain
\bea  \vvvert \theta(g)-\theta(\tilde g ) \vvvert &\le  \Lambda_s ^k \vvvert P_s h_{-k} - P_s \tilde h_{-k} \vvvert + C\eps_{gap} \sum_{j=1}^k \Lambda_s^{j-1} \vvvert h_{-j} -\tilde h_{-j}\vvvert  \\
& \le \left( \frac{\Lambda_s}{\Lambda_-}\right)^k + \frac{C\eps_{gap}}{\Lambda_-} \sum _{j=1}^k  \left( \frac{\Lambda_s}{\Lambda_-}\right)^{j-1}\vvvert \{h_k\}-\{\tilde h_k\} \vvvert_{\Lambda_-,\Lambda_+},
\eea
and   computing the limit on the right-hand side, we find
\[
\vvvert \theta(g)-\theta(\tilde g ) \vvvert 
\le  C\frac{\eps_{gap}}{\Lambda_- -\Lambda_s}\vvvert \{h_k\}-\{\tilde h_k\} \vvvert_{\Lambda_-,\Lambda_+} .
 \]
Since $\vvvert \{h_k\}-\{\tilde h_k\} \vvvert_{\Lambda_-,\Lambda_+} \le \frac{1}{1-K} \vvvert g -\tilde g\vvvert$ by Lemma \ref{lem-lipTheta}, this finishes the proof.  
   
\end{proof}

\begin{definition}[Center manifold] The center manifold is the image of $\Theta _0:E_c \rightarrow L^2_{p+1}$ defined by 
\[W^\eps _c:= \textrm{graph}_{E_c}\theta =\{f+\theta(f)\,:\, f\in E_c \} .\]
\end{definition}
The following property explains the origin of its name, invariant manifold.
\begin{lemma} [Invariance] \label{L23} There holds
\[S^\eps (W^\eps _c)= W^\eps _c .\]
\begin{proof} Note that we have $h \in W^\e _c$ if and only if there exists an orbit $\{h_k\}\in \ell _{\Lambda_-,\Lambda_+}$ with $h_0=h$. Here an orbit means     $S^\eps (h_{j-1})=h_j$ for all $j\in \mathbb{Z}$.  

If $h \in W^\eps _c$ and $\{h_k\}$ is the corresponding orbit, then a time translation $\{h_{k-1}\}$ is also an orbit in $\ell _{\Lambda_-,\Lambda_+}$. Thus  $h_{-1}\in W^\e _c $, i.e., $h=S^\e (h_{-1})\in S^\e (W^\e _c)$. This shows $W^\e _c\subset S^\e (W^\eps _c)$. Similarly, if $h \in S^\e (W^\e _c)$ there exist an orbit $\{g_k\}\in \ell_{\Lambda_+,\Lambda_-}$ such that $h= S^\e (g_0)=g_1$. Since $\{g_{k+1}\}$ is an orbit in the same space, this shows $h=g_1 \in W^\e _c$. 
\end{proof}

\end{lemma}

Lemma \ref{L23} shows $S^\eps _n(W^\eps_c)=W^\eps_c$ for all $n\in\N$. In fact, it holds  $S^\eps_t(W^\eps_c)=W^\eps_c$ for non-integer times $t\ge0$ as well, and hence a solution starting from $W^\e_c$ stays in it for later times. Indeed, for non-integer $t$, the proof is essentially identical to Lemma \ref{L23} but the only missing ingredient is to show if $\{h_k\}_{k \in \mathbb{Z}}$ is an orbit in $\ell _{\Lambda_-,\Lambda_+}$ then so is $\{S^\eps_t(h_k)\}_{k\in \mathbb{Z}}$. This follows from the stability inequality \eqref{eq-stability} in Theorem \ref{thm-truncated}. An interesting further question to study is if $\theta$ is a $C^1$-map which is tangent to $E_c$ at the origin, i.e., $D\theta (0)=0$. Notice that the previous result only proves $\theta(0)=0$. {The stronger statement
is indeed  expected and a potential proof relies on suitable maximal regularity estimates, see, for instance, the constructions in \cite{Seis15+,SeisWinkler22+}. As this property is not essential for our purposes and maximal regularity estimates are not (yet) available, we do not investigate the regularity of the manifold here. } 

\subsection{Stable manifolds}

Now we construct the stable manifold with the help of another  fixed point argument. 

 For notational simplicity, we will ignore the dependence of $\eps$ in the proof.  Define the operator 
\[I: E_s \times  (L^2_{p+1})^{\mathbb{N}_0}\to (L^2_{p+1})^{\mathbb{N}_0}\] with $I(\cdot,\cdot)= \{ I_k(\cdot,\cdot) \} _{k \in \mathbb{N}_0}$ by 
\be \label{eq-defI} I_k(g_s, \{h_i\}):= \begin{cases} \ba 
 P_s (g_s )+ L_c^{-1}P_c(h_1-R(h_0)) &&\text{ if } k =0 \\
 P_s S(h_{k-1})+L_c^{-1} P_c (h_{k+1} -R (h_k))&&\text{ if }k\ge 1 .
   \ea \end{cases}\ee  
For given $g\in L^2_{p+1}$ and $g_s\in E_s$, we consider the map $I_{g,g_s}$ defined by  \[ \{h_i\}_{i\in \mathbb{N}_0} \mapsto I(g_s+P_s(g),\{h_i\} + \{S^i(g)\}) - \{S^i(g) \}.\]

\begin{lemma}\label{L10}
Let $\eps_{gap}$ be fixed so that the condition \eqref{eq-egap} is satisfied. Suppose the truncation parameter $\eps$ is small so that the condition \eqref{eq-R-eps-gap} is met.  The mapping $I_{g,g_s}$ admits a unique fixed point $I_{g,g_s}(\{h_k\}) = \{h_k\}$ and there holds
\be \vvvert I_{g,g_s}(\{h_k\})\vvvert_{\Lambda_-} \le    \vvvert g_s\vvvert . \ee 
\end{lemma}

\begin{proof}
Observe that we have contraction property: for all $g_s\in E_s$ and $\{h_i\}$, $\{\tilde h_i\}$ in $(L^2_{p+1})^{\mathbb{N}_0}$ with $\{h_i-\tilde h_i \} \in \ell _{\Lambda_-}$, there holds 
\bea \label{eq-Icontraction}  \vvvert I(g_s, \{h_i\})-I(g_s,\{ \tilde h_s \}) \vvvert _{\Lambda_-}\le  K \vvvert \{h_i\}-\{\tilde h _i \}\vvvert _{\Lambda_-}. \eea  
The proof proceeds very similar to the contraction estimate in Lemma \ref{thm-centermfd}. We skip the details.
This property shows $I_{g,g_s}$ maps $\ell _{\Lambda_-}$ into itself and it is a contraction mapping. Therefore, there is a unique fixed point to the operator.

Next, we refine above and show $I_{g,g_s}$ is  a map from $B_{\vvvert g_s \vvvert}$ into itself. Observe that 
\[I(g_s+P_s(g), \{S^i (g)\})- \{S^i(g)\} = \{ g_s \delta_{l0}\}_{l\in \mathbb{N}_0} .\]
Let $I_{g,g_s}= \{(I_{g,g_s})_i\}_{i\in \mathbb{N}_0}$. For $k\ge1$, it holds that
\bea  &\vvvert (I_{g,g_s})_k(\{h_i\})\vvvert=\vvvert I_k (g_s+P_s g, \{h_i\}+\{S^i(g)\}) - I(g_s+P_s g, \{S^i(g)\}) \vvvert   \\
&\le K \Lambda_-^k \vvvert \{h_i\} \vvvert_{\Lambda_-} \eea 
If $k=0$, we have \bea(I_{g,g_s})_0(\{h_i\}) = [I_0(g_s+P_s g, \{h_i \}+\{S^i(g)\}) - I_0(g_s+P_s g, \{S^i(g)\})] + g_s,\eea which provides the orthogonal decomposition in $L^2_{p+1}=E_c \oplus E_s$, and thus  
\bea \vvvert (I_{g,g_s})_0(\{h_i\})\vvvert &\le \max (K \vvvert \{h_i\}\vvvert_{\Lambda_-}, \vvvert g_s \vvvert ).
\eea  
In conclusion, by combining two cases, 
\be \vvvert I_{g,g_s}(\{h_k\})\vvvert_{\Lambda_-} \le \max (K\vvvert \{h_k\} \vvvert_{\Lambda_-}, \vvvert g_s\vvvert ), \ee 
showing $I_{g,g_s}$ maps  $B_{\vvvert g_s \vvvert}$ in $\ell_{\Lambda_-}$ to itself. Moreover, this implies the fixed point should lie in $B_{\vvvert g_s \vvvert }$. 
\end{proof}

Let $\{h_i\}\in \ell_{\Lambda_-}$ be the unique fixed point to the problem $I_{g,g_s}$. Then by construction, $\{h_i\}+\{S^i(g)\}$ is a semi-flow starting from $h_0+g$. Let us denote this map that assigns $\{h_i\}$ for each $g_s$ by $\Psi_g$. That is, $\Psi_g : E_s \to \ell_{\Lambda_-}$ is such that $\Psi_g (g_s)$ is the unique fixed point to $I_{g,g_s}$. 

\begin{lemma}\label{L21}
$ \Psi_g$ is a Lipschitz map from $(L^2_{p+1}\cap E_s, \vvvert \cdot \vvvert)$ to $(\ell _{\Lambda_- }, \vvvert \cdot \vvvert_{\Lambda_-})$ with $\mathrm{Lip}(\Theta)\le \frac1{1-K}$. 
\end{lemma}
\begin{proof}
This follows simply by construction, similar to the Lipschitz continuity of $\Theta$ in Lemma \ref{lem-lipTheta}.
\end{proof}

We now define $\psi_g:E_s \to L^2_{p+1}$ as the time zero slice of $\Psi_g $ 
\[\psi_g (g_s) := (\Psi_g (g_s) )_0.\]By its construction, we obtain that $\mathrm{Lip}\, (\psi_g) \le \frac{1}{1-K}$. Again, we are able to provide a better bound.

\begin{lemma}\label{lem-lippsi} The mapping $\psi: (L^2_{p+1}\otimes 
E_s,  \vvvert \cdot \vvvert\otimes \vvvert \cdot \vvvert) \to (E_c\cap L^2_{p+1},\vvvert \cdot \vvvert)  $ given by $\psi^\eps (g,g_s):=\psi^\eps _g(g_s)$ is continuous and $\psi_g : (E_s,  \vvvert \cdot \vvvert) \to (E_c\cap L^2_{p+1},\vvvert \cdot \vvvert)  $  is uniformly Lipschitz  
with $\mathrm{Lip}(\psi_g) \lesssim \e _{gap}$.
\end{lemma}

\begin{proof}The argument for Lipschitz continuity is similar to the one of Lemma \ref{lem-liptheta}.
Let $\{h_k\}= \Psi_g(g_s)$ and $\{\tilde h _k \} = \Psi_g (\tilde g_s)$ be fixed points of $I_{g,g_s}$ and $I_{g,\tilde g_s}$, respectively. By the definition of $I$ in \eqref{eq-defI}, 
\[\vvvert P_c(\tilde h_k - h_k ) \vvvert \le \frac{1}{\Lambda_c} \vvvert P_c (\tilde h_{k+1}- h_{k+1})\vvvert + \frac{\e _{gap}}{\Lambda_c} \vvvert \tilde h_k - h_k \vvvert .\] By repeating the above $m$ times from $k=0$, for each $m\ge1$, we have 
\bea \vvvert P_c(\tilde h_0 - h_0 )\vvvert &\le \frac{1}{\Lambda_c ^m } \vvvert P_c (\tilde h_m - h_m )\vvvert + \sum_{k=0}^{m-1} \frac{\e _{gap}}{\Lambda_c ^{k+1}} \vvvert \tilde h_k - h_k \vvvert\\
&\le  \left[ \frac{\Lambda_-^m}{\Lambda_c ^m } + \sum_{k=0}^{m-1} \frac{\e_{gap}\Lambda^k_- }{\Lambda^{k+1}_c} \right ]\vvvert \{\tilde h_k\}-\{ h_k \} \vvvert_{\Lambda_-} .
\eea
The right-hand side converges as $m$ tends to infinity, and we find
\[
\vvvert P_c(\tilde h_0 - h_0 )\vvvert  \le   \frac{\e _{gap}}{\Lambda_c - \Lambda_-} \vvvert \{\tilde h_k\}-\{ h_k \} \vvvert_{\Lambda_-}  .
\]
Since $\mathrm{Lip}\, (\Psi_g)\le \frac{1}{1-K}$, this shows $\mathrm{Lip}\, (\psi_g) \le \frac{\e _{gap}}{(1-K)(\Lambda_c - \Lambda_-)}.$

Next, the continuity of $\psi(\cdot,\cdot)$ follows if we show the continuity in $g$ for each fixed $g_s$ (since $\psi(g,g_s)$ is uniformly Lipschitz in $g_s$). This is a consequence of the well-posedness of the time-one-map $S$ and the characterization of the stable manifold $M_g$. More precisely, suppose $g_i \to g$ in $L^2_{p+1}$ and $q_s\in E_s$ is fixed. Since $\Vert \psi_{g_i}(q_s) \Vert_{L^2_{p+1}} \le  \frac{1}{1-K} \Vert q_s\Vert_{L^2_{p+1}} $ and $E_c$ is finite dimensional, every subsequence of $g_i$ has a further subsequence $g_{n_j}$ so that $\psi_{g_{n_i}}(q_s)=\psi(g_{n_i},q_s)$   converges strongly in $L^2_{p+1}$ to some $q_c\in E_c $. By the well-posedness of $S$, for each integer $k\ge0$, \[S^k (g_{n_i})\to S^k(g) \text{ and }S^k(g_{n_i}+g_s+\psi_{g_{n_i}}(q_s))\to S^k (g+q_s+q_c), \]in $L^2_{p+1}$. Thus the inequality  
\[\vvvert S^k (g_{n_i}) - S^k(g_{n_i} +q_s +\psi_{g_n{i}}(q_s))\vvvert  \le \Lambda_-^k \vvvert q_s +\psi_{g_{n_i}}(q_s)\vvvert \]
descends to 
\[\vvvert S^k (g) -S^k(g+q_s+q_c) \vvvert \le \Lambda^k_- \vvvert q_s+q_c \vvvert. \]
Finally by the characterization of $M_{g}$, this implies $q_c = \psi_{g}(q_s)$, i.e., $\psi(g_i,q_s )\to \psi(g,q_s)$ as $i \to \infty$. 
\end{proof}

With these preparations, we can define the stable manifold that has the desired properties.

\begin{definition}[Stable manifold] The stable manifold with respect to an element $g\in L^2_{p+1}$ is the translated graph of $\psi_g$ over the stable subspace, defined by
\[M^\eps _g := g+ \{g_s + \psi^\e _g (g_s)\,:\, g_s \in E_s \}.
\] 
\end{definition}

We provide an important characterization.

\begin{lemma}\label{L24}
$M^\eps _g$ can be characterized as 
\bea \label{eq-characterizationstable} M^\eps _g &= \{ \tilde g \in L^2_{p+1}\,:\, \sup _{k\in \mathbb{N}_0}\Lambda_- ^{-k} \vvvert S^k  (g) - S^k (\tilde g) \vvvert\le \vvvert P_s(g- \tilde g) \vvvert  \}\\ 
& = \{ \tilde g \in L^2_{p+1}\,:\, \sup _{k\in \mathbb{N}_0}\Lambda_- ^{-k} \vvvert S^k  (g) - S^k (\tilde g) \vvvert<\infty  \}.
\eea 
\end{lemma}

\begin{proof}
First, if $\tilde g \in M_g$ is given, then $ \tilde g = g+ g_s +\psi_g(g_s)$ for some $g_s \in E_s$ and in view of the  construction of $\psi_g$, it holds
\[\vvvert \{S^k (\tilde g)\}  - \{S^k (g) \} \vvvert \le \vvvert g_s \vvvert = \vvvert P_s(\tilde g -g) \vvvert . \] 
Conversely, if  $\tilde g=(g + q_s)+q_c$ for some $q_s \in E_s$ and $q_c \in E_c$ satisfies
\[\vvvert \{ S^k(\tilde g)\} - \{  S^k(g) \}\vvvert_{\Lambda_-} < \infty ,\] then it is direct to check that $\{h_k\}=\{S^k (\tilde g) \} - \{S^k(g)\} \in \ell _{\Lambda_-}$ is a fixed point of $I_{g,q_s}$ in $\ell _{\Lambda_-}$, which is unique. Thus $q_c=\psi_g(q_s)$ and $\vvvert \{h_k\}\vvvert_{\Lambda_-}\le \vvvert q_s\vvvert=\vvvert P_s(\tilde g -g)\vvvert .$
\end{proof}

We finally have to show that the stable manifold constitutes a foliation over the center manifold.

\begin{lemma}\label{L22}
 If $\eps_{gap}>0$ is sufficiently small, then $M^\eps_g $ and the center manifold $W^\eps _c$ have a unique intersection point, i.e., $\{M^\eps _g\}_{g\in L^2_{p+1}}$ is a foliation of $L^2_{p+1}$ over $W^\eps _c$. 
\end{lemma}

\begin{proof}
We start by noting that by the definition of $W_c^\e $ and $M_g$, a point  $q = q_s+ q_c$ in $L^2_{p+1}$ is an intersection point if and only if $q_s \in E_s$ is a fixed point of the mapping $\chi:E_s\to E_s $ defined by
\[\chi (q_s):= \theta (\psi_g  (q_s -P_s g ) +P_c g).\]  
Since both $\theta$ and $\psi_g$ are Lipschitz continuous functions with constants of the order $\eps_{gap}$ by the virtue of Lemmas \ref{lem-liptheta} and \ref{lem-lippsi}, it follows that $\chi$ is Lipschitz with $\mathrm{Lip}\, (\chi ) \le C^2 \e ^2_{gap}$. Therefore $\chi$ is a contraction and has thus   a unique fixed point for sufficiently small $\e_{eps}$.
\end{proof}

\subsection{Proof of main theorem}
By combining the invariant manifold constructions obtained above, we now prove that Theorem \ref{thm-finiteapprox}  shows  that the center manifold $W^\eps_c$  of the truncated equation becomes a finite $K$-dimensional approximation of solutions to the original equation. 


\begin{proof}[Proof of Theorem \ref{thm-finiteapprox}]
We choose $\eps_{gap}$ and $\eps$ so that all hypotheses of the previous lemmas and of Theorem \ref{thm-truncated} are satisfied.

Let  $h(t)$ be a given solution to the original equation \eqref{eq-relativeerror}. By Lemma \ref{L11}, $h(t)$ becomes a solution to the $\eps$-truncated equations. By a time translation,   we may without loss of generality assume that this holds for $t\ge 0$.  

Thanks to Lemma \ref{L22} there exist a unique point $\tilde h_0$ in the intersection $ W^\eps_c  \cap M_{h_0}$. Because $W_c^{\eps}$ is invariant by Lemma \ref{L23}, the approximate $\tilde h(t)$ given by Theorem \ref{thm-truncated} remains on that manifold (initially for discrete times but it can be continuously extended to all positive times). In view of the characterization of Lemma \ref{L24} it follows that $\tilde h(t)$ approximates the original solution $h(t)$
in the sense that 
\[ \vvvert \tilde h_\eps (t) - h(t) \vvvert  \le  \Lambda_-^t \vvvert \tilde h_\eps(0)- h(0) \vvvert  \quad \text{ for all } t\in \mathbb{N}_0.\]  Next, since the solution to the truncated equation \eqref{eq-trunc-eq} depends continuously on the initial datum with respect to the Hilbert space topology (see \eqref{eq-stability} in Theorem \ref{thm-truncated}), this extends to all times,
\[ \vvvert \tilde h_\eps (t) - h(t) \vvvert  \le C  e^{-\lambda_-t }  \vvvert \tilde h_\eps(0)- h(0) \vvvert  \quad \text{ for all } t\in [0,T].\] 
This finishes the proof.

\end{proof}

\section*{Acknowledgments}
The first author has been partially supported by  National Research Foundation of Korea grants No. 2022R1C1C1013511 and No. RS-2023-00219980 funded by the Korea government(MSIT), and POSTECH Basic Science Research Institute grant No. 2021R1A6A1A10042944.
The work of the second author is funded by the Deutsche Forschungsgemeinschaft (DFG, German Research Foundation) under Germany's Excellence Strategy EXC 2044 --390685587, Mathematics M\"unster: Dynamics--Geometry--Structure. Both authors acknowledge support by the DFG through the international collaboration grant 493768209. The authors thank the anonymous referees who kindly reviewed and provided valuable suggestions.

\bibliography{fastdiffusion}
\bibliographystyle{abbrv}

\end{document}